\documentclass[twoside]{amsart}

\usepackage{amsmath,amsfonts,amsthm,mathrsfs}
\usepackage{amssymb}

\usepackage[unicode,bookmarks]{hyperref} 
\usepackage[usenames,dvipsnames]{xcolor}
\hypersetup{colorlinks=false}

\usepackage{expl3}

\ExplSyntaxOn
\prop_new:N \g_cite_map_prop
\tl_new:N \l_citekey_result_tl

\cs_new:Npn \mapcitekey #1#2 {
  \clist_map_inline:nn {#2}
       {  \prop_gput:Nnn  \g_cite_map_prop  {##1} {#1}   }
}

\cs_new:Npn \getcitekey #1 {
   \prop_get:NoN \g_cite_map_prop{#1}  \l_citekey_result_tl
   \quark_if_no_value:NF \l_citekey_result_tl
       {  \tl_set_eq:NN #1  \l_citekey_result_tl  }
}

\cs_new:Npn \showcitekeymaps {\prop_show:N  \g_cite_map_prop }
\ExplSyntaxOff

\usepackage{etoolbox}
\makeatletter
\patchcmd{\@citex}{\if@filesw}{\getcitekey\@citeb \if@filesw}%
    {\typeout{*** SUCCESS ***}}{\typeout{*** FAIL ***}}
\patchcmd{\nocite}{\if@filesw}{\getcitekey\@citeb \if@filesw}%
    {\typeout{*** SUCCESS ***}}{\typeout{*** FAIL ***}}
\makeatother

\usepackage{enumitem}

\usepackage{chngcntr}

\usepackage{tikz}
\usepackage{tikz-cd}
\tikzset{commutative diagrams/arrow style=Latin Modern}

\newcommand{\derR}{\mathbf{R}}

\newcommand{\tensor}{\otimes}

\newcommand{\NN}{\mathbb{N}}

\newcommand{\menge}[2]{\bigl\{ \thinspace #1 \thinspace\thinspace \big\vert%
\thinspace\thinspace #2 \thinspace \bigr\}}

\DeclareMathOperator{\im}{im}

\DeclareMathOperator{\id}{id}

\DeclareMathOperator{\Supp}{Supp}
\DeclareMathOperator{\codim}{codim}

\DeclareMathOperator{\End}{End}
\DeclareMathOperator{\Hom}{Hom}

\DeclareMathOperator{\Pic}{Pic}

\newcommand{\define}[1]{\emph{#1}}

\newcommand{\gl}{\mathfrak{gl}}

\newcommand{\shf}[1]{\mathscr{#1}}
\newcommand{\OX}{\shf{O}_X}

\newcommand{\defeq}{\underset{\textrm{def}}{=}}

\newcommand{\restr}[1]{\big\vert_{#1}}

\def\overbar#1#2#3{{%
	\setbox0=\hbox{\displaystyle{#1}}%
	\dimen0=\wd0
	\advance\dimen0 by -#2 
	\vbox {\nointerlineskip \moveright #3 \vbox{\hrule height 0.3pt width \dimen0}%
		\nointerlineskip \vskip 1.5pt \box0}%
}}

\newcommand{\fu}{f^{\ast}}
\newcommand{\fl}{f_{\ast}}

\newcommand{\iu}{i^{\ast}}

\newcommand{\pu}{p^{\ast}}
\newcommand{\pl}{p_{\ast}}
\newcommand{\qu}{q^{\ast}}

\newcommand{\tl}{t_{\ast}}
\newcommand{\gu}{g^{\ast}}

\newcommand{\varphiu}{\varphi^{\ast}}
\newcommand{\varphil}{\varphi_{\ast}}

\newcommand{\shF}{\shf{F}}
\newcommand{\shG}{\shf{G}}

\newcommand{\shO}{\shf{O}}

\newcommand{\shFm}{\shF}

\makeatletter
\let\@@seccntformat\@seccntformat
\renewcommand*{\@seccntformat}[1]{%
  \expandafter\ifx\csname @seccntformat@#1\endcsname\relax
    \expandafter\@@seccntformat
  \else
    \expandafter
      \csname @seccntformat@#1\expandafter\endcsname
  \fi
    {#1}%
}
\newcommand*{\@seccntformat@subsection}[1]{%
  \textbf{\csname the#1\endcsname.}
}
\makeatother

\makeatletter
\let\@paragraph\paragraph
\renewcommand*{\paragraph}[1]{%
	\vspace{0.3\baselineskip}%
	\@paragraph{\textit{#1}}%
}
\makeatother

\counterwithin{equation}{subsection}
\counterwithout{subsection}{section}
\counterwithin{figure}{subsection}

\newtheorem{theorem}[equation]{Theorem}
\newtheorem*{theorem*}{Theorem}
\newtheorem{lemma}[equation]{Lemma}
\newtheorem*{lemma*}{Lemma}
\newtheorem{corollary}[equation]{Corollary}
\newtheorem*{corollary*}{Corollary}
\newtheorem{proposition}[equation]{Proposition}
\newtheorem*{proposition*}{Proposition}

\newtheorem*{conjecture*}{Conjecture}

\theoremstyle{definition}
\newtheorem{definition}[equation]{Definition}
\newtheorem*{definition*}{Definition}
\theoremstyle{remark}
\newtheorem*{remark}{Remark}

\newtheorem{example}[equation]{Example}
\newtheorem*{example*}{Example}
\newtheorem*{problem*}{Problem}
\newtheorem*{note}{Note}

\theoremstyle{plain}

\newcommand{\theoremref}[1]{\hyperref[#1]{Theorem~\ref*{#1}}}
\newcommand{\lemmaref}[1]{\hyperref[#1]{Lemma~\ref*{#1}}}
\newcommand{\definitionref}[1]{\hyperref[#1]{Definition~\ref*{#1}}}
\newcommand{\propositionref}[1]{\hyperref[#1]{Proposition~\ref*{#1}}}
\newcommand{\conjectureref}[1]{\hyperref[#1]{Conjecture~\ref*{#1}}}
\newcommand{\corollaryref}[1]{\hyperref[#1]{Corollary~\ref*{#1}}}
\newcommand{\exampleref}[1]{\hyperref[#1]{Example~\ref*{#1}}}
\newcommand{\exerciseref}[1]{\hyperref[#1]{Exercise~\ref*{#1}}}

\makeatletter
\let\old@caption\caption
\renewcommand*{\caption}[1]{%
	\setcounter{figure}{\value{equation}}%
	\stepcounter{equation}%
	\old@caption{#1}\relax%
}
\makeatother

\newcounter{intro}

\newtheorem{intro-conjecture}[intro]{Conjecture}
\newtheorem{intro-corollary}[intro]{Corollary}
\newtheorem{intro-theorem}[intro]{Theorem}

\newcommand{\OA}{\mathscr{O}_A}

\newcommand{\omY}{\omega_Y}

\newcommand{\OY}{\shO_Y}

\newcommand{\parref}[1]{\hyperref[#1]{\S\ref*{#1}}}
\newcommand{\chapref}[1]{\hyperref[#1]{Chapter~\ref*{#1}}}

\makeatletter
\newcommand*\if@single[3]{%
  \setbox0\hbox{${\mathaccent"0362{#1}}^H$}%
  \setbox2\hbox{${\mathaccent"0362{\kern0pt#1}}^H$}%
  \ifdim\ht0=\ht2 #3\else #2\fi
  }
\newcommand*\rel@kern[1]{\kern#1\dimexpr\macc@kerna}
\newcommand*\widebar[1]{\@ifnextchar^{{\wide@bar{#1}{0}}}{\wide@bar{#1}{1}}}
\newcommand*\wide@bar[2]{\if@single{#1}{\wide@bar@{#1}{#2}{1}}{\wide@bar@{#1}{#2}{2}}}
\newcommand*\wide@bar@[3]{%
  \begingroup
  \def\mathaccent##1##2{%
    \if#32 \let\macc@nucleus\first@char \fi
    \setbox\z@\hbox{$\macc@style{\macc@nucleus}_{}$}%
    \setbox\tw@\hbox{$\macc@style{\macc@nucleus}{}_{}$}%
    \dimen@\wd\tw@
    \advance\dimen@-\wd\z@
    \divide\dimen@ 3
    \@tempdima\wd\tw@
    \advance\@tempdima-\scriptspace
    \divide\@tempdima 10
    \advance\dimen@-\@tempdima
    \ifdim\dimen@>\z@ \dimen@0pt\fi
    \rel@kern{0.6}\kern-\dimen@
    \if#31
      \overline{\rel@kern{-0.6}\kern\dimen@\macc@nucleus\rel@kern{0.4}\kern\dimen@}%
      \advance\dimen@0.4\dimexpr\macc@kerna
      \let\final@kern#2%
      \ifdim\dimen@<\z@ \let\final@kern1\fi
      \if\final@kern1 \kern-\dimen@\fi
    \else
      \overline{\rel@kern{-0.6}\kern\dimen@#1}%
    \fi
  }%
  \macc@depth\@ne
  \let\math@bgroup\@empty \let\math@egroup\macc@set@skewchar
  \mathsurround\z@ \frozen@everymath{\mathgroup\macc@group\relax}%
  \macc@set@skewchar\relax
  \let\mathaccentV\macc@nested@a
  \if#31
    \macc@nested@a\relax111{#1}%
  \else
    \def\gobble@till@marker##1\endmarker{}%
    \futurelet\first@char\gobble@till@marker#1\endmarker
    \ifcat\noexpand\first@char A\else
      \def\first@char{}%
    \fi
    \macc@nested@a\relax111{\first@char}%
  \fi
  \endgroup
}
\makeatother

\AtBeginDocument{%
   \def\MR#1{}
}

\usepackage[all,cmtip]{xy}
\newcommand\sO{{\mathcal O}}

\newcommand\sI{{\mathcal I}}

\renewcommand{\shG}{\mathscr{G}}
\renewcommand{\shF}{\mathscr{F}}
\newcommand{\shQ}{\mathscr{Q}}
\newcommand{\shGt}{\widetilde{\shG}}
\renewcommand{\gl}{g_{\ast}}

\newcommand{\Kb}{K^{\bullet}}
\newcommand{\omXm}{\omX^{\tensor m}}
\newcommand{\omX}{\omega_X}
\newcommand{\omXY}{\omega_{X/Y}}
\newcommand{\omXYm}{\omXY^{\tensor m}}
\newcommand{\shI}{\mathcal{I}}

\renewcommand{\gl}{g_{\ast}}

\newcommand{\D}{{\bf D}}

\renewcommand{\restr}[1]{\vert_{#1}}

\begin{document}

\title[Pushforwards of pluricanonical bundles]%
	{Pushforwards of pluricanonical bundles under morphisms to abelian varieties}
	
\author[L.~Lombardi]{Luigi Lombardi}
\address{Dipartimento di Matematica ``Federigo Enriques'', Università Statale di Milano, 
Via Cesare Saldini, 20133 Milano, Italy}
\email{luigi.lombardi@unimi.it}

\author[M.~Popa]{Mihnea Popa}
\address{Department of Mathematics, Northwestern University,
2033 Sheridan Road, Evanston, IL 60208, USA} 
\email{mpopa@math.northwestern.edu}

\author[Ch.~Schnell]{Christian Schnell}
\address{Department of Mathematics, Stony Brook University, Stony Brook, NY 11794-3651}
\email{christian.schnell@stonybrook.edu}

\begin{abstract}
Let $f \colon X \to A$ be a morphism from a smooth projective variety to an abelian
variety (over the field of complex numbers). We show that the sheaves $\fl
\omXm$ become globally generated after pullback by an isogeny.  We use this to deduce
a decomposition theorem for these sheaves when $m \ge 2$, analogous to that obtained
by Chen-Jiang when $m = 1$. This is in turn applied to effective results for
pluricanonical linear series on irregular varieties with canonical singularities.

\end{abstract}
\date{\today}
\maketitle
\section{Introduction}
 \subsection{Results}

There are many results in the literature about pushforwards of canonical bundles
under morphisms to abelian varieties. The purpose of this paper is to extend these
results to the case of pluricanonical bundles.

Let us briefly summarize what is known about pushforwards of canonical
bundles. Let $f \colon X \to A$ be any morphism from a smooth projective variety $X$
to an abelian variety $A$, defined over the field of complex numbers. The pushforward sheaf
$\fl \omX$ is a coherent sheaf on $A$ with the following remarkable
properties:
\begin{enumerate}
\item By \cite{GL2,Simpson} (cf. also \cite[Corollary~17.1]{PPS}), the cohomology support loci
\[
	V_{\ell}^i(A, \fl \omX) = \menge{\alpha \in \Pic^0(A)}%
		{\dim H^i(A, \alpha \tensor \fl \omX) \geq \ell}
\]
are finite unions of torsion subvarieties (= abelian subvarieties translated by
torsion points) of the dual abelian variety $\Pic^0(A)$.
\item By \cite{GL1,Hacon:GV}, $\fl \omX$ is a GV-sheaf, meaning that
\[
	\codim_{\Pic^0(A)} V^i(A, \fl \omX) \geq i 
\]
for all $i \geq 0$, where $V^i(A, \fl \omX): = V^i_1(A, \fl \omX)$; see also \S\ref{scn:FM}.
\item By \cite{ChenJiang:positivity,PPS}, one has a canonical decomposition
\[
	\fl \omX \cong \bigoplus_{i \in I} \bigl( \alpha_i \tensor \qu_i \shF_i \bigr),
\]
into pullbacks of M-regular coherent sheaves (see Definition 3.1) $\shF_i$ from quotients $q_i \colon A
\to A_i$ of the abelian variety, tensored by torsion line bundles $\alpha_i \in
\Pic^0(A)$.
\item By \cite{GL2,LPS}, the Fourier-Mukai transform of $\fl \omX$, which is a complex of
coherent sheaves on the dual abelian variety $\Pic^0(A)$, is locally analytically
quasi-isomorphic to a linear complex.
\end{enumerate}

\begin{note}
The same is true for the higher direct images of the canonical bundle: the properties
above are inherited by direct summands, and by a theorem of Koll\'ar, the sheaf $R^i
\fl \omX$ is isomorphic to a direct summand in the pushforward of the canonical
bundle from the intersection of $i$ sufficiently ample hypersurfaces in $X$.
\end{note}

The question that motivated this work is what happens for pushforwards of
pluricanonical bundles $\fl \omXm$ with $m \geq 2$. Throughout the paper, we work
over the field of complex numbers. It was already known
that $\fl \omXm$ is always a GV-sheaf \cite[Theorem~1.10]{pluricanonical}, and that the $0$-th
cohomology support loci $V_{\ell}^0(A, \fl \omXm)$ are always finite unions of torsion
subvarieties (cf. \cite[Theorem~3.5]{Lai} and \cite[Theorem~10.1]{HPS}). Our
main result is that, in fact, \emph{all} the
results about pushforwards of canonical bundles carry over to the case of
pluricanonical bundles.  More precisely, we prove the following theorem.

\begin{intro-theorem} \label{thm:main-summand}
Let $f \colon X \to A$ be a morphism from a smooth projective variety to an abelian
variety. Then for each $m \geq 2$, there is a
generically finite morphism $g_m \colon X_m \to X$ from a smooth projective variety
$X_m$, 
\[
\begin{tikzcd}
X_m \rar{g_m} \arrow[bend right=40]{rr}{f \circ g_m} & X \rar{f} & A
\end{tikzcd}
\]
such that $\fl \omXm$ is isomorphic to a direct summand in $(f \circ g_m)_{\ast}
\omega_{X_m}$.
\end{intro-theorem}

In particular, properties (i)-(iv) above continue to hold: all cohomology support loci of $\fl \omXm$ are finite unions of
torsion subvarieties; the Fourier-Mukai transform of $\fl \omXm$ is locally computed
by a linear complex; $\fl \omXm$ decomposes into a sum of pullbacks of M-regular
sheaves tensored by torsion line bundles; etc. (See \S\ref{scn:consequences} for
a detailed discussion.)

\begin{note}
Let us briefly explain why one should expect Theorem \ref{thm:main-summand} to be true.
Suppose for simplicity that our smooth projective variety $X$ is a good
minimal model, meaning that some power of the canonical bundle $\omX$ is globally
generated. Fix an integer $N \geq m$ such that $\omX^{\tensor N}$ is globally
generated; then
\[
	\omX^{\tensor N} \cong \OX(D),
\]
where $D$ is the divisor of a general section, and therefore irreducible and smooth.
Let $g \colon Y \to X$ be the branched covering obtained by extracting
$N$-th roots; then
\[
	\gl \omY \cong \bigoplus_{k=1}^N \omX^{\tensor k},
\]
and so $\fl \omXm$ is isomorphic to a direct summand in $(f \circ g)_{\ast} \omY$.
(One can still get the same conclusion under the weaker assumption that $X$ admits a good
minimal model with at worst terminal singularities; the abundance conjecture predicts
that good minimal models exist for all smooth projective varieties that are not
uniruled.)
\end{note}

Fortunately, we need much less information about $\omX$ to draw this conclusion.
More precisely, we obtain Theorem \ref{thm:main-summand} as a consequence of the
following result, which roughly says that pushforwards of pluricanonical bundles are
``almost'' globally generated.  For $m = 1$ and generically finite $f$, this is due
to Chen and Jiang (see \cite[Theorem~4.1]{ChenJiang:positivity} and its proof); for $m = 1$ and
arbitrary $f$, it follows from \cite[Theorem~A]{PPS}; we recall the argument in
\S\ref{scn:m=1} below.

\begin{intro-theorem} \label{thm:main}
Let $f \colon X \to A$ be a morphism from a smooth projective variety to an abelian
variety. Then there exists an isogeny $\varphi \colon A'
\to A$ such that $\fl' \omega_{X'}^{\tensor m} \cong \varphiu \bigl( \fl \omXm
\bigr)$ is globally generated for every $m \geq 1$.
\[
\begin{tikzcd}
X' \dar{f'} \rar & X \dar{f} \\
A' \rar{\varphi} & A
\end{tikzcd}
\]
\end{intro-theorem}

Once we know that $\fl \omXm$ becomes globally generated after pulling back by an
isogeny, we can use Viehweg's cyclic covering trick to show that it is isomorphic to
a direct summand in the pushforward of a canonical bundle (as in
Theorem \ref{thm:main-summand}); in fact, this is really a by-product of our proof of
Theorem \ref{thm:main}.

One consequence of Theorem \ref{thm:main-summand}, already alluded to above, is the
following decomposition theorem for pushforwards of pluricanonical bundles under
morphisms to abelian varieties. For $m = 1$, the existence of such a decomposition
was proved by Chen and Jiang \cite{ChenJiang:positivity} when $f$ is generically finite, and subsequently
in \cite{PPS} in general; the search for a generalization to the case $m \geq 2$ was
our starting point for this work.

\begin{intro-theorem}\label{thm:CJ}
In the setting of Theorem \ref{thm:main-summand}, there exists a finite decomposition
\[
	 \fl \omXm \cong \bigoplus_{i \in I} \bigl( \alpha_i \tensor \pu_i \shF_i \bigr),
\]
where each $p_i \colon A \to A_i$ is a quotient morphism with connected fibers to an
abelian variety, each $\shF_i$ is an M-regular coherent sheaf on $A_i$, and each
$\alpha_i \in \Pic^0(A)$ is a line bundle that becomes trivial when pulled back by
the isogeny in Theorem \ref{thm:main}.
\end{intro-theorem}

Since GV-sheaves are nef, whereas M-regular sheaves are ample
\cite[Corollary~3.2]{Debarre}, one can think of Theorem \ref{thm:CJ} as saying that
$\fl \omXm$ is semi-ample. With more work, one can prove a stronger positivity result
for pushforwards of pluricanonical bundles under the Albanese morphism, valid only
for $m \geq 2$. (This fits the general principle that, for any morphism $f \colon X \to
Y$ of smooth projective varieties, the sheaves $\fl \omXYm$ with $m \geq 2$
behave more uniformly than $\fl \omXY$.)

\begin{intro-theorem}\label{thm:complete_intro}
Let $f \colon X \to A_X$ be the Albanese morphism of a smooth projective variety, 
and let $m\geq 2$ be an integer. 
If $X \to Z$ is a smooth model of the Iitaka fibration of $X$, and $\psi \colon A_X \to A_Z$ is the 
induced morphism between the Albanese varieties, then there exists a finite direct sum decomposition
$$f_*\omega_X^{\otimes m} \; \cong \;\bigoplus_{i\in I} \big(\alpha_i \otimes \psi^*\shF_i\big),$$ 
where $\shF_i$ are coherent sheaves on $A_Z$ satisfying $\mathrm{IT}_0$ (see Definition 3.1), and $\alpha_i \in \Pic^0 (X)$ 
are torsion line bundles whose order can be bounded independently of $m$.
\end{intro-theorem}

The proof of Theorem \ref{thm:main} relies on an inductive procedure in terms of the
dimension of the target abelian variety $A$, using along the way the structure of the
$0$-th cohomological support locus of $\omega_X^{\otimes m}$, invariance of
plurigenera, and the fact that continuously globally generated sheaves become
globally generated after base change. We also crucially use new analytic results, 
introduced
into algebraic geometry by the recent work of Cao and P\u{a}un \cite{CP}, namely the
existence of positively curved singular hermitian metrics on pushforwards of relative
pluricanonical bundles. As a consequence of the Ohsawa-Takegoshi $L^2$-extension
theorem, these metrics have a ``minimal extension property'', which allows one to
split off trivial quotient line bundles; see \cite{HPS} for a detailed discussion
of this circle of ideas.

Regularity results for direct images of canonical bundles to abelian varieties, loosely meaning various versions 
of the GV or M-regularity property, 
can be used for studying  effective basepoint freeness and very ampleness statements for
pluricanonical linear series on smooth irregular varieties. When working with
singular varieties however, the canonical sheaf is typically not a line bundle, and
it becomes necessary to use pluricanonical bundles on resolutions instead. We give a
few such effective statements in \S\ref{scn:effective}, as an application of
Theorem \ref{thm:complete_intro}.  The general result is Theorem \ref{gg-index}; here
we only exemplify with a more easily stated result in the case of maximal Albanese
dimension, extending facts from \cite{PP4,PP2}.

\begin{intro-corollary}\label{cor:mad}
Let $Y$ be a normal projective variety of general type with at worst canonical singularities, and let $N$ be the Cartier index of $K_Y$, assumed to be $\ge 2$.  If the Albanese map $g\colon Y\rightarrow A$ is generically finite onto its image, then 
outside of the exceptional locus (the union of positive dimensional fibers) of $g$, $\shO_Y(2NK_Y)$ is globally generated 
and $\shO_Y(3NK_Y)$ is very ample.
\end{intro-corollary}

Note that this result is also true when $X$ is smooth (and so $N= 1$), as proved with considerable extra work in \cite{JLT}; the methods used here only give a weaker bound
when $N =1$. For a further discussion of the context regarding such results, please see \S\ref{scn:effective}.

\section{Preliminaries}

\subsection{Fourier-Mukai transform and generic vanishing}\label{scn:FM}

Let $A$ be an abelian variety, and $\widehat{A} \cong \Pic^0 (A)$ its dual. We denote
by 
\[
	\derR\widehat{\mathcal{S}}\colon \D(A)\rightarrow \D(\widehat{A}), \qquad
		\derR\widehat{\mathcal{S}} \shF \defeq \derR {p_2}_* (p_1^*\shF \otimes P) 
\]
the Fourier-Mukai functor induced by a normalized Poincar\'{e} bundle $P$ on
$A\times \widehat{A}$, where $p_1$ and $p_2$ are the projections onto $A$ and
$\widehat{A}$ respectively. We recall the following three conditions, ordered by
increasing strength:

\begin{definition}
A coherent sheaf $\shF$ on $A$ is said to:
\begin{enumerate}
\item be a \emph{$\mathrm{GV}$-sheaf} if ${\rm codim}~{\rm Supp}~R^i \widehat{\mathcal{S}} \shF \ge i$ for all $i > 0$.
\item be \emph{M-regular} if ${\rm codim}~{\rm Supp}~R^i \widehat{\mathcal{S}}\shF > i$ for all $i > 0$.
\item \emph{satisfy $\mathrm{IT}_0$} if $R^i \widehat{\mathcal{S}} \shF = 0$ for all $i > 0$.
\end{enumerate}
\end{definition}  

Denoting by 
$$V^i (A, \shF) \defeq \menge{\alpha \in \Pic^0(A)}{%
		\dim H^i(A, \shF \tensor \alpha) \neq 0}$$
the $i$-th cohomological support locus of $\shF$, it is a well-known consequence of standard base change
arguments that the conditions in the definition are equivalent to the same statements with $R^i \widehat{\mathcal{S}} \shF$
replaced by $V^i(A, \shF)$; of course, the $\mathrm{IT}_0$ condition means in this
case that $V^i(A, \shF)$ is empty for $i > 0$.

\begin{lemma}\label{Mreg_positive}
If $\shF \neq 0$ is an M-regular sheaf on an abelian variety $A$, then the Euler
characteristic $\chi (A, \shF) > 0$, and in particular $V^0(A, \shF) = \Pic^0 (A)$.
\end{lemma}
\begin{proof}
The first assertion is \cite[Lemma~5.1]{PP5}. Since $\chi (A, \shF) = \chi (A, \shF \otimes \alpha)$ for all $\alpha \in \Pic^0 (A)$, the second follows from the vanishing of $H^i (A, \shF\otimes \alpha)$ for all $i > 0$ and $\alpha$ general.
\end{proof}

\begin{proposition}\label{locfreeh0}
Let $\shF$ be a  $\mathrm{GV}$-sheaf on an abelian variety $A$. If $\dim H^0(A,\shF\otimes
\alpha)$ is independent of $\alpha\in \Pic^0(A)$, then $\shF$ satisfies
$\mathrm{IT}_0$, and $\derR\widehat{\mathcal{S}}(\shF)$ is locally free.
\end{proposition}
\begin{proof}
Denote by $\derR\Delta$ the functor $\derR \mathcal{H}om_{A}(-,\shO_{A})$. By \cite[Corollary~3.10]{PP3} 
(cf. also \cite[Theorem 2.3]{PP2} for the shape of the statement used here) the complex $\derR \widehat{\mathcal{S}} \derR\Delta \shF$ is concentrated in degree $\dim A$, so we can set 
$ \derR \widehat{\mathcal{S}} \derR\Delta \shF\cong \shG[-\dim A]$ for some sheaf $\shG$ on $\widehat{A}$.  Using the sequence of isomorphisms
$$H^0(A,\shF\otimes \alpha)\cong \Hom_A(\derR\Delta \shF, \alpha)\cong \Hom_{\D(\widehat{A})}\big(\derR \widehat{\mathcal{S}}\derR\Delta \shF ,\derR\widehat{\mathcal{S}} \alpha \big)\cong $$ 
$$\cong \Hom_{\widehat{A}}(\shG,\shO_{\alpha^{-1}})\cong \Hom_{\mathbf{C}}\big(\shG\otimes \mathbf{C}(\alpha^{-1}),\mathbf{C} \big)$$
we see that the dimension of the fibers $\shG\otimes \mathbf{C}(\alpha)$ is constant for all $\alpha \in A$. It follows that $\shG$ is locally free and hence an $\infty$-syzygy sheaf.
The conclusion  follows now by  \cite[Theorem 5.4]{PP1}.
\end{proof}

\subsection{The Chen-Jiang decomposition property}\label{scn:CJ}

Chen and Jiang \cite[Theorem~1.1]{ChenJiang:positivity} have proved a decomposition theorem for the
pushforward of the canonical bundle under a generically finite morphism to an abelian
variety. This was extended to higher direct images via arbitrary morphisms from compact K\"ahler manifolds to 
compact complex tori in \cite[Theorem~A]{PPS}. In this section, we record a few basic observations about such
decompositions.

\begin{definition}
We say that a coherent sheaf $\shF$ on an abelian variety $A$ has the
\define{Chen-Jiang decomposition property} if $\shF$ admits a finite direct sum
decomposition
\begin{equation} \label{eq:CJ}
	\shF \cong \bigoplus_{i \in I} \bigl( \alpha_i \tensor \pu_i \shF_i \bigr),
\end{equation}
where each $A_i$ is an abelian variety, each $p_i \colon A \to A_i$ is a surjective
morphism with connected fibers, each $\shF_i$ is a nonzero M-regular coherent sheaf on
$A_i$, and each $\alpha_i \in \Pic^0(A)$ is a line bundle of finite order.
\end{definition}

\begin{remark}
By the projection formula, $\shF_i$ is a direct summand in 
\[
	(p_i)_{\ast} \bigl( \shF \tensor \alpha_i^{-1} \bigr).
\]
If $\shF = \fl \omX$ for a morphism $f \colon X \to A$ from a smooth projective
variety $X$, then the coherent sheaf $\shF_i$ is a direct summand in $(p_i \circ
f)_{\ast} \bigl( \omega_X \tensor \fu \alpha_i^{-1} \bigr)$, and therefore again
isomorphic to a direct summand in the pushforward of the canonical bundle from a
finite \'etale covering of $X$.
\end{remark}

\begin{lemma} \label{lem:CJ-V0}
If $\shF$ admits a Chen-Jiang decomposition as in \eqref{eq:CJ}, then
\[
	V^0(A, \shF) = \bigcup_{i \in I} \alpha_i^{-1} \tensor
		\im \Bigl( \pu_i \colon \Pic^0(A_i) \to \Pic^0(A) \Bigr).
\]
\end{lemma}

\begin{proof}
This is an easy computation, using the projection formula and the fact that \sloppy $V^0(A_i,
\shF_i) = \Pic^0(A_i)$ (by Lemma \ref{Mreg_positive}).
\end{proof}

In any Chen-Jiang decomposition as in \eqref{eq:CJ}, we can clearly collect terms
whose Fourier-Mukai transform has the same support, and arrange that 
\[
	\alpha_i^{-1} \tensor \pu_i \Pic^0(A_i) \neq \alpha_j^{-1} \tensor \pu_j \Pic^0(A_j)
\]
for $i \neq j$. We say that a Chen-Jiang decomposition is \define{reduced} if it has
this additional property. In the reduced case, we can define a partial ordering
$\leq$ on the index set $I$ by declaring that
\begin{equation} \label{eq:PO}
	i \leq j \quad \text{if and only if} \quad
		\alpha_i^{-1} \tensor \pu_i \Pic^0(A_i) 
			\subseteq \alpha_j^{-1} \tensor \pu_j \Pic^0(A_j).
\end{equation}
In more concrete terms, $i \leq j$ means that the morphism $p_i \colon A \to A_i$
factors through $p_j \colon A \to A_j$, and that $\alpha_j \tensor \alpha_i^{-1} \in
\pu_j \Pic^0(A_j)$. 

\begin{lemma} \label{lem:CJ}
Every reduced Chen-Jiang decomposition has the property that
\[
	\Hom \bigl( \alpha_j \tensor \pu_j \shF_j, 
		\alpha_i \tensor \pu_i \shF_i \bigr) = 0
\]
unless $j \leq i$.
\end{lemma}

\begin{proof}
This is the content of \cite[Corollary~2.7]{ChenJiang:positivity}.
\end{proof}

In the remainder of this section, we prove that the existence of a Chen-Jiang
decomposition is preserved under passing to direct summands. 

\begin{proposition} \label{prop:CJ-summand}
Let $\shF$ be a coherent sheaf on an abelian variety $A$. If $\shF$ has the
Chen-Jiang decomposition property, and if 
\[
	\shF \cong \shF' \oplus \shF'',
\]
then both $\shF'$ and $\shF''$ also have the Chen-Jiang decomposition property.
\end{proposition}

Before we begin the proof, a small remark about projectors. Suppose that $U$ is an
object in an abelian category, and $P \in \End(U)$ is an endomorphism with $P \circ P
= P$. Then we get a direct sum decomposition
\[
	U = \ker(\id - P) \oplus \ker P.
\]
The projection to the first summand is given by $P$, the projection to the second
summand by $(\id - P)$.

\begin{proof}
Let us consider the endomorphism $P \in \End(\shF)$ given by projecting to the
summand $\shF'$. It satisfies $P \circ P = P$, and it is easy to see that
\[
	\shF' = \ker(\id - P) \quad \text{and} \quad
	\shF'' = \ker P.
\]
With respect to a reduced Chen-Jiang decomposition of $\shF$, the endomorphism $P$ is
represented by a collection of morphisms
\[
	P_{i,j} \in \Hom \bigl( \alpha_j \tensor \pu_j \shF_j, 
		\alpha_i \tensor \pu_i \shF_i \bigr),
\]
and Lemma \ref{lem:CJ} shows that $P_{i,j} = 0$ unless $j \leq i$. From $P \circ
P = P$, we obtain
\[
	P_{i,j} = \sum_{k \in I} P_{i,k} \circ P_{k,j}
		= \sum_{j \leq k \leq i} P_{i,k} \circ P_{k,j},
\]
due to the fact that $(I, \leq)$ is partially ordered. Specializing to $i = j$, we
find that
\[
	P_{i,i} = \sum_{k \in I} P_{i,k} \circ P_{k,i} = P_{i,i} \circ P_{i,i}
\]
is itself a projector. Because the morphism $p_i \colon A \to A_i$ has connected
fibers, the projection formula gives
\[
	P_{i,i} \in \Hom \bigl( \alpha_i \tensor \pu_i \shF_i, 
		\alpha_i \tensor \pu_i \shF_i \bigr) =
	\Hom(\shF_i, \shF_i),
\]
and so $P_{i,i}$ determines a direct sum decomposition $\shF_i = \shF_i' \oplus
\shF_i''$, where again 
\[
	\shF_i' = \ker(\id - P_{i,i}) \quad \text{and} \quad \shF_i'' = \ker P_{i,i}.
\]
Clearly $\shF_i'$ and $\shF_i''$ are both M-regular. In terms of $\shF$, this gives
us a more refined direct sum decomposition
\begin{equation} \label{eq:decompositions}
	\shF \cong \bigoplus_{i \in I} \bigl( \alpha_i \tensor \pu_i \shF_i' \bigr)
		\oplus \bigoplus_{i \in I} \bigl( \alpha_i \tensor \pu_i \shF_i'' \bigr).
\end{equation}

To finish the proof, we are going to construct an automorphism of $\shF$ that 
takes $\shF'$ and $\shF''$ to the two summands in \eqref{eq:decompositions}.
To that end, let us define a ``diagonal'' endomorphism $D \in \End(\shF)$ by setting
\[
	D_{i,j} = \begin{cases}
		P_{i,i} &\text{if $i = j$,} \\
		0 &\text{if $i \neq j$.}
	\end{cases}
\]
Clearly, $D \circ D = D$, and the two summands in \eqref{eq:decompositions} are
nothing but $\ker(\id - D)$ and $\ker D$. Now consider the endomorphism $E = \id - D
- P \in \End(\shF)$. Lemma \ref{lem:automorphism} below shows that $E$ is an
automorphism of $\shF$. It is easy to see that
\[
	 E \circ P = - D \circ P = D \circ E;
\]
and $E$ therefore maps the subsheaf $\shF' = \ker(\id - P)$ isomorphically to the
subsheaf $\ker(\id - D)$, and the subsheaf $\shF'' = \ker P$ isomorphically to the
subsheaf $\ker D$. Consequently,
\[
	\shF' \cong \bigoplus_{i \in I} \bigl( \alpha_i \tensor \pu_i \shF_i' \bigr),
\]
which is what we needed to prove.
\end{proof}

\begin{lemma} \label{lem:automorphism}
$E$ is an automorphism of $\shF$.
\end{lemma}

\begin{proof}
We have to show that the system of equations $y = E x$ has a unique solution for
every collection of local sections $y_i \in \alpha_i \tensor \pu_i \shF_i$.
Concretely, the equations are
\[
	y_i = x_i - 2 P_{i,i} x_i - \sum_{j < i} P_{i,j} x_j \qquad \text{($i \in I$).}
\]
If $i \in I$ is minimal, the condition is $y_i = x_i - 2 P_{i,i} x_i$, which has the
unique solution $x_i = y_i - 2 P_{i,i} y_i$. In general, we have
\[
	x_i = (\id - 2 P_{i,i}) y_i + (\id - 2 P_{i,i}) \sum_{j < i} P_{i,j} x_j,
\]
and since $I$ is finite, we can argue by induction.
\end{proof}

\subsection{Generic base change}

In this section, we use Grothendieck's generic flatness theorem to prove a sort of
``generic base change theorem''. Suppose that $f \colon X \to Y$ and $p \colon Y \to Z$ are
proper morphisms between schemes of finite type over a field, with $Z$ generically
reduced and $g = p \circ f$ surjective. 
\[
\begin{tikzcd}
X_z \rar[hook] \dar{f_z} & X \dar{f} \arrow[bend left=60]{dd}{g} \\
Y_z \rar[hook] \dar & Y \dar{p} \\
\{z\} \rar[hook] & Z
\end{tikzcd}
\]
Let $\shF$ be a coherent sheaf on $X$. For any point $z \in Z$, we denote by $f_z
\colon X_z \to Y_z$ the restriction of the morphism $f \colon X \to Y$ to the fibers
over $z$, and by $\shF_z$ the restriction of $\shF$ to $X_z$.

\begin{proposition} \label{prop:base-change}
Notation and assumptions being as above, there is a nonempty Zariski-open subset $U
\subseteq Z$ such that the base change morphism
\begin{equation} \label{eq:base-change}
	R^i \fl \shF \restr{Y_z} \cong R^i (f_z)_{\ast} \shF_z
\end{equation}
is an isomorphism for every $z \in U$ and every $i \in \NN$.
\end{proposition}

\begin{proof}
Since $Z$ is generically reduced, we can restrict everything to a nonempty
Zariski-open subset in $Z$ and reduce the problem to the case where $Z$ is
nonsingular. By the theorem on generic flatness \cite[Th\'eor\`eme~6.9.1]{EGA},
$\shF$ is flat over a nonempty Zariski-open subset of $Z$. Replacing $Z$ by this
subset, we reduce the problem to
the case where $\shF$ is flat over $Z$. The higher direct image sheaves $R^i \fl \shF$
are coherent, and vanish for $i \gg 0$. Appealing to the theorem of generic flatness
once more, we see that $Y$ and all the coherent sheaves $R^i \fl \shF$ are flat over
a common nonempty Zariski-open subset of $Z$; restricting to this subset, we can
assume that this holds over $Z$. After all these reductions have been made, we claim
that \eqref{eq:base-change} is now true for all points of $Z$.

Let $\shO_z$ denote the structure sheaf of a point $z \in Z$. Because $Z$ is
nonsingular, we can (on some affine neighborhood) resolve $\shO_z$ by a Koszul
complex $\Kb$ consisting of locally free sheaves; the augmented complex 
\[
	\dotsb \to K^{-1} \to K^0 \to \shO_z \to 0
\]
is exact. Now $\shF$ is flat over $Z$, and therefore the complex
\[
	\dotsb \to \shF \tensor_{\OX} \gu K^{-1} 
		\to \shF \tensor_{\OX} \gu K^0
		\to \shF \tensor_{\OX} \gu \shO_z \to 0
\]
stays exact; in other words, $\shF \tensor_{\OX} \gu \Kb$ is a resolution of $\shF_z$,
viewed as a coherent sheaf on $X$. This gives
\[
	\derR \fl \shF \tensor_{\OY} \pu \Kb 
		\cong \derR \fl \Bigl( \shF \tensor_{\OX} \gu \Kb \Bigr)
		\cong \derR (f_z)_{\ast} \shF_z 
\]
by the projection formula. Each $R^i \fl \shF$ is also flat over $Z$, and so for the
same reason as above, the complex $R^i \fl \shF \tensor_{\OY} \pu \Kb$ has cohomology
only in degree zero, and is therefore a resolution of $R^i \fl \shF \restr{Y_z}$,
viewed as a coherent sheaf on $Y$. We now obtain \eqref{eq:base-change} by
taking cohomology in degree $i$.
\end{proof}

\section{Proof of the main theorem}

\subsection{Pushforwards of canonical bundles}\label{scn:m=1}

In this section, we briefly indicate how the proof of Theorem \ref{thm:main} in the
case $m = 1$ follows quickly from \cite[Theorem~A]{PPS} (and of course from  
\cite[Theorem~1.1]{ChenJiang:positivity} when $f$ is generically finite). 

\begin{theorem} \label{thm:main-1}
Let $f \colon X \to A$ be a morphism from a smooth projective variety to an abelian
variety. If a coherent sheaf $\shF$ on $A$ is isomorphic to a direct summand of $\fl
\omX$, then there is an isogeny $\varphi \colon A' \to A$ such that
$\varphiu \shF$ is globally generated.  
\end{theorem}

\begin{proof}
According to the results in \S\ref{scn:CJ}, there is a decomposition
\[
	\shF \cong \bigoplus_{i \in I} \bigl( \alpha_i \tensor \pu_i \shF_i \bigr),
\]
where each $p_i \colon A \to A_i$ is a surjective morphism with connected fibers
to an abelian variety $A_i$, each $\shF_i$ is an M-regular coherent sheaf on $A_i$,
and each $\alpha_i \in \Pic^0(A)$ has finite order. 

Now every M-regular coherent sheaf is continuously globally generated
\cite[Proposition~2.13]{PP4}, and according to a result by Debarre
\cite[Proposition~3.1]{Debarre}, continuously globally generated sheaves on
abelian varieties become globally generated after an isogeny. Since the line bundles
$\alpha_i$ have finite order, we can therefore find an isogeny $\varphi \colon A' \to
A$ that makes the pullback of every $\alpha_i$ trivial, and the pullback of each
$\pu_i \shF_i$ globally generated. But then the pullback of $\shF$ is globally
generated as well.
\end{proof}

When dealing with the case $m\ge 2$, we use precisely the opposite approach: we first
show directly the analogue of Theorem \ref{thm:main-1}, and then use it in order to
deduce the decomposition in Theorem \ref{thm:CJ}.

\subsection{Strategy of the proof}

In this section, we outline our strategy for proving Theorem \ref{thm:main}. Let $f
\colon X \to A$ be a morphism from a smooth projective variety $X$ to an abelian
variety $A$, and for any $m \geq 1$, denote by $\fl \omXm$ the pushforward of the
$m$-th pluricanonical bundle to a coherent sheaf on the abelian variety.

Note first that it is enough to prove the statement for each $m \ge 1$ individually.
Indeed, the fact that there exists a single isogeny that works for all $m \geq 1$
then follows from the fact that the sheaf of $\shO_A$-algebras 
\[	
	\bigoplus_{m \ge 0} f_* \omega_X^{\otimes m} 
\]
is finitely generated \cite[Theorem~1.2]{BCHM}. Fix an integer $m \geq 1$, and define
\[
	\shFm \defeq \fl \omXm. 
\]
If $\shF$ is zero, then the assertion is trivially true; we will therefore assume
from now on that $\shF \neq 0$.

A useful observation is that we are allowed to base change by arbitrary isogenies for
the purpose of proving Theorem \ref{thm:main} and Theorem \ref{thm:main-summand}.

\begin{lemma} \label{lem:isogeny}
Let $\varphi \colon A' \to A$ be an isogeny, and let $f' \colon X' \to A'$ denote the
base change of the morphism $f \colon X \to A$, as in the diagram below.
\[
\begin{tikzcd}
X' \dar{f'} \rar & X \dar{f} \\
A' \rar{\varphi} & A
\end{tikzcd}
\]
Then the conclusion of Theorem \ref{thm:main} (resp.\ Theorem \ref{thm:main-summand}) holds
for $f' \colon X' \to A'$ if and only if it holds for $f \colon X \to A$.
\end{lemma}

\begin{proof}
In the case of Theorem \ref{thm:main}, the point is that the coherent sheaf $\shFm =
\fl \omXm$ pulls back to $\shFm' = \fl' \omega_{X'}^{\tensor m}$, because the morphism
from $X'$ to $X$ is finite \'etale. In the case of Theorem \ref{thm:main-summand}, the point
is that $\shFm$ occurs as a direct summand in 
\[
	\varphil \shFm' \cong \varphil \varphiu \shFm;
\]
thus if $\shFm'$ is isomorphic to a direct summand of $(f' \circ g')_{\ast}
\omega_{Y'}$ for some generically finite morphism $g' \colon Y' \to X'$, then
the induced morphism $g \colon Y' \to X$ is still generically finite, and $\shFm$ is
isomorphic to a direct summand of $(f \circ g)_{\ast} \omega_{Y'}$, 
\end{proof}

We know from \cite[Theorem~1.10]{pluricanonical} that $\shFm$ is a
$\mathrm{GV}$-sheaf on $A$. We also know that
all irreducible components of the locus
\[
	V_{\ell}^0(A, \shFm)
		= \menge{\alpha \in \Pic^0(A)}{\dim H^0(A, \shFm \tensor \alpha) \geq \ell}
\]
are torsion subvarieties, for any $\ell \in \NN$; see \cite[Theorem~3.5]{Lai} for $\ell =1$, and \cite[Theorem~10.1]{HPS}
and (a special case of) \cite[Theorem~1.3]{Shibata2}  in general. The idea behind the proof of
Theorem \ref{thm:main} is to consider the maximal subsheaf of $\shFm$ that becomes
globally generated after base change by an isogeny, and then to show that this
maximal subsheaf equals $\shFm$. More precisely, we define $\shG \subseteq
\shFm$ as the image of the evaluation morphism
\begin{equation} \label{eq:shG-def}
	\bigoplus_{\substack{\alpha \in \Pic^0(A)\\\text{torsion}}} 
		H^0(A, \shFm \tensor \alpha) \tensor \alpha^{-1}  \to \shFm.
\end{equation}
The following lemma is obvious from the definition of
$\shG$; we shall prove later on that the statement is actually true for every
$\alpha \in \Pic^0(A)$.

\begin{lemma}  \label{lem:GF-torsion}
One has 
\[
	H^0(A, \shG \tensor \alpha) = H^0(A, \shFm \tensor \alpha)
\]
for every torsion line bundle $\alpha \in \Pic^0(A)$. 
\end{lemma}

Observe that, because $\shFm$ is coherent, by the noetherian property there must be finitely many torsion
line bundles $\alpha_1, \dotsc, \alpha_N \in \Pic^0(A)$ such that
\[
	\bigoplus_{i=1}^N H^0(A, \shFm \tensor \alpha_i) \tensor \alpha_i^{-1} \to \shG
\]
is surjective. In particular, $\shG$ is again a coherent sheaf of $\shO_A$-modules. After making a base change by a suitable isogeny -- which is
permissible by Lemma \ref{lem:isogeny} -- we may therefore assume without loss of
generality that $\shG$ is actually globally generated. Now Theorem \ref{thm:main} is
claiming that $\shG = \shFm$, and so we consider the quotient sheaf $\shQ =
\shFm/\shG$, and the resulting short exact sequence
\begin{equation} \label{eq:shQ}
	0 \to \shG \to \shFm \to \shQ \to 0.
\end{equation}
Our proof that $\shQ = 0$ proceeds along the following lines:
\begin{enumerate}
\item We use Viehweg's cyclic covering trick to show that $\shG$ occurs as a direct
summand in the pushforward of a canonical bundle (from a smooth projective variety,
generically finite over $X$). This implies, among other things, that $\shQ$
is a $\mathrm{GV}$-sheaf, and so $\shQ = 0$ is equivalent to $V^0(A, \shQ)$ being empty; see \cite[Lemma~7.4]{HPS}.
\item We show that $H^0(A, \shG\otimes \alpha) = H^0(A, \shFm \otimes \alpha)$ for
every $\alpha \in
\Pic^0(A)$.  After base change by an isogeny, this implies that $V^0(A, \shQ)
\subseteq V^1(A, \shG)$ is contained in a finite union of subtori of codimension
$\geq 1$.  
\item We show that $V^0(A, \shQ)$ does not meet any positive-dimensional
subtorus contained in $V^1(A, \shG)$. Here the point is that the ``generic base
change theorem'', together with invariance of plurigenera, reduces the problem to an
instance of Theorem \ref{thm:main} over an abelian variety of lower dimension, which
can be handled by induction.
\item The only remaining case is $V^1(A, \shG) = \{\OA\}$. Here we use some
recent work about singular hermitian metrics on pushforwards of relative
pluricanonical bundles \cite{CP,HPS} to prove that $V^0(A, \shQ)$ must be empty.
\end{enumerate}

\subsection{Viehweg's covering trick}

Recall that we defined $\shG$ as the maximal subsheaf of $\shFm = \fl \omXm$ that
becomes globally generated after an isogeny; by the reduction above we can assume 
that it is itself globally generated. The first step in
the proof is to show that $\shG$ occurs as a direct summand in the pushforward of a
canonical bundle. This is a simple consequence of Viehweg's cyclic covering trick,
using the maximality and  global generation of $\shG$.

\begin{proposition} \label{prop:viehweg}
Keeping the notation from above, there is a smooth projective variety $Y$, and a
generically finite morphism $g \colon Y \to X$, with the property that $\shG$ is
isomorphic to a direct summand in $(f \circ g)_{\ast} \omY$.
\end{proposition}

\begin{proof}
Following Viehweg \cite[\S5]{Viehweg}, we consider the adjoint morphism
\[
	\fu \shG \to \omXm.
\]
After blowing up on $X$, if necessary, we can assume that the image sheaf is of the
form $\omXm \tensor \OX(-E)$ for a simple normal crossing divisor $E$. Because the pullback
sheaf $\fu \shG$ is still globally generated, the line bundle
\[
	\omXm \tensor \OX(-E)
\]
is also globally generated. It is therefore isomorphic to $\OX(D)$, where $D$ is a
smooth (possibly reducible) divisor, not contained in the support of $E$, such that $D + E$
still has simple normal crossings. We have arranged that
\[
	\omXm \cong \OX(D+E),
\]
and so we can take the associated branched covering and resolve singularities. This
gives us a generically finite morphism $g \colon Y \to X$ of degree $m$.

Here comes the main point. By a result of Esnault and Viehweg
\cite[Lemma~2.3]{Viehweg}, the pushforward $\gl \omY$ contains as a direct summand
the sheaf
\[
	\omX \tensor \omX^{\tensor(m-1)} 
		\tensor \OX \left(- \left\lfloor \tfrac{m-1}{m} (D+E) \right\rfloor \right).
\]
Because $D$ is smooth and not contained in the support of $E$, it follows that
\[
	\omXm \tensor \OX \left(- \left\lfloor \tfrac{m-1}{m} E \right\rfloor \right)
\]
is a direct summand in $\gl \omY$. If we now apply $\fl$, we find that
\[
	\shGt \defeq \fl \Bigl( \omXm \tensor 
		\OX \left(- \left\lfloor \tfrac{m-1}{m} E \right\rfloor \right) \Bigr)
\]
is a direct summand in $(f \circ g)_{\ast} \omY$. As such, it becomes globally
generated after an isogeny (by Theorem \ref{thm:main-1}), and so the
morphism
\begin{equation} \label{eq:shGt}
	\bigoplus_{\substack{\alpha \in \Pic^0(A)\\\text{torsion}}} 
		H^0 \bigl( A, \shGt \tensor \alpha \bigr) \tensor \alpha^{-1}  \to \shGt
\end{equation}
is surjective. 

To continue, we observe that $\shG \subseteq \shGt \subseteq \shFm$. Indeed, since $\fl$ is left-exact, 
we have inclusions
\[
	\fl \Bigl( \omXm \tensor \OX(-E) \Bigr) \subseteq 
	\fl \Bigl( \omXm \tensor 
		\OX \left(- \left\lfloor \tfrac{m-1}{m} E \right\rfloor \right) \Bigr) 
		\subseteq \fl \omXm,
\]
and by construction, the inclusion of $\shG$ into $\shFm$ factors
through the subsheaf on the left, hence also through $\shGt$. For any torsion
point $\alpha \in \Pic^0(A)$, we now obtain
\[
	H^0(A, \shG \tensor \alpha) \subseteq H^0 \bigl( A, \shGt \tensor \alpha \bigr)
		\subseteq H^0(A, \shFm \tensor \alpha),
\]
and because the two spaces on the left and right are equal by
Lemma \ref{lem:GF-torsion}, we conclude from the definition of $\shG$ in
\eqref{eq:shG-def} and the surjectivity of \eqref{eq:shGt} that $\shG = \shGt$. In
particular, $\shG$ is isomorphic to a direct summand in $(f \circ g)_{\ast} \omY$.
\end{proof}

\subsection{Cohomology support loci}

The fact that $\shG$ is isomorphic to a direct summand in the pushforward of a canonical
bundle has several nice consequences. First, $\shG$ is also a $\mathrm{GV}$-sheaf (which is
not obvious from its definition). Secondly, all of its cohomology support loci
\[
	V_{\ell}^i(A, \shG) = \menge{\alpha \in \Pic^0(A)}{%
		\dim H^i(A, \shG \tensor \alpha) \geq \ell}
\]
are finite unions of torsion subvarieties (for any $\ell \in \NN$). This allows us to
strengthen the result in Lemma \ref{lem:GF-torsion}.

\begin{lemma} \label{lem:GF}
One has $H^0(A, \shG \tensor \alpha) = H^0(A, \shFm \tensor \alpha)$ for every $\alpha
\in \Pic^0(A)$.
\end{lemma}

\begin{proof}
This is true for all torsion points $\alpha \in \Pic^0(A)$ by
Lemma \ref{lem:GF-torsion}, and torsion points are dense inside $V_{\ell}^0(A, \shG)$
and $V_{\ell}^0(A, \shFm)$.
\end{proof}

Because $\shG$ and $\shFm$ are both $\mathrm{GV}$-sheaves, $\shQ = \shFm/\shG$ is also a $\mathrm{GV}$-sheaf.
In order to show that $\shQ = 0$, it is therefore enough to prove that $V^0(A, \shQ)$
is empty.  Tensoring by $\alpha \in \Pic^0(A)$ and taking cohomology, we obtain
\[
	0 \to H^0(A, \shG \tensor \alpha) \to H^0(A, \shFm \tensor \alpha)
		\to H^0(A, \shQ \tensor \alpha) \to H^1(A, \shG \tensor \alpha).
\]
The first two spaces on the left are equal (by Lemma \ref{lem:GF}), and so
\begin{equation} \label{eq:V0}
	V^0(A, \shQ) \subseteq V^1(A, \shG)
\end{equation}
is contained in a finite union of torsion subvarieties of codimension $\geq 1$.

\begin{remark}
Besides the containment in \eqref{eq:V0}, we do not know, a priori, anything else 
about the possible irreducible components of $V^0(A, \shQ)$.
\end{remark}

By \cite[Lemma~7.4]{HPS} and the containment \eqref{eq:V0} we may suppose that \sloppy $V^1(A,\shG)\neq \emptyset$.
Since we are allowed to make a base change by a suitable isogeny, we may assume
without loss of generality that all irreducible components of $V^1(A, \shG)$ contain
the point $\OA \in \Pic^0(A)$. Now there are two cases: either $V^1(A, \shG) =
\{\OA\}$; or $V^1(A, \shG)$ contains a positive-dimensional subtorus of $\Pic^0(A)$.
The following proposition deals with the first case.

\begin{proposition} \label{prop:origin}
If $V^1(A, \shG) = \{\OA\}$, then $\shQ = 0$.
\end{proposition}

\begin{proof}
Since $\shQ$ is a $\mathrm{GV}$-sheaf, and $V^0(A, \shQ) \subseteq \{\OA\}$ by \eqref{eq:V0}, we
can say immediately that $\shQ$ is either zero, or a unipotent vector bundle of
positive rank; see \cite[Proposition~7.5]{HPS}. 
We shall
argue that the second possibility leads to a contradiction. Indeed, if $\shQ$ is a
unipotent vector bundle of positive rank, one can find a surjective morphism $\shQ
\to \OA$. By composing it with the projection from $\shFm$ to $\shQ$, we obtain a
surjective morphism 
\[
	\shFm \to \OA. 
\]
In particular, $\Supp \shFm = A$, and so $f \colon X \to A$
is surjective. Because the canonical bundle of $A$ is trivial, we have $\shFm \cong
f_* \omega_{X/A}^{\tensor m}$. As explained in \cite[Theorem~27.1]{HPS}, the
torsion-free sheaf $\shFm$ therefore has a singular hermitian metric with
semi-positive curvature and the ``minimal extension property''. The existence of such
a metric forces the morphism $\shFm \to \OA$ to split \cite[Corollary~27.2 (b)]{HPS}.
Taking global sections, we find that
\[
	H^0(A, \shFm) \to H^0(A, \shQ) \to H^0(A, \OA)
\]
is surjective -- but this contradicts the fact that $H^0(A, \shG) = H^0(A, \shFm)$.
The conclusion is that $\shQ = 0$, at least in the special case where $V^1(A, \shG) =
\{\OA\}$.
\end{proof}

\subsection{Components of positive dimension}

It remains to deal with the second case, namely that $V^1(A, \shG)$ contains a
positive-dimensional subtorus of $\Pic^0(A)$. Such a subtorus is the image of the
pullback morphism $\pu \colon \Pic^0(B) \to \Pic^0(A)$, where $p \colon A \to B$ is a
surjective morphism with connected fibers to an abelian variety $B$. Note that we
have $1 \leq \dim B \leq \dim A - 1$, due to the fact that $\shG$ is a $\mathrm{GV}$-sheaf.
After a base change by an isogeny on $B$ -- which is permissible by
Lemma \ref{lem:isogeny} -- we may assume that $A = B \times F$ is a product of two
abelian varieties $B$ and $F$, and that $p \colon B \times F \to B$ is the first
projection. To simplify the notation, set $g = p \circ f$, as in the following
diagram:
\[
\begin{tikzcd}
X \dar{f} \arrow[bend left=60]{dd}{g} \\
B \times F \dar{p} \\
B
\end{tikzcd}
\]
\begin{lemma}
In this situation, we have $\pl \shG = \pl \shFm$.
\end{lemma}

\begin{proof}
By induction on the dimension of the abelian variety $A$, Theorem \ref{thm:main} is true for the sheaf $\pl \shFm = \gl \omXm$,
and because
\[
	H^0(B, \pl \shG \tensor \beta) = H^0(A, \shG \tensor \pu \beta)
		= H^0(A, \shFm \tensor \pu \beta) = H^0(B, \pl \shFm \tensor \beta)
\]
for every $\beta \in \Pic^0(B)$, we find that $\pl \shG = \pl \shFm$.
Note that the base step of the induction is satisfied by Proposition \ref{prop:origin}. In fact, when $A$ is one-dimensional,  the locus $V^1(A,\shF) $ consists of at most the trivial sheaf $\shO_A$ up to the action of a  base change by a suitable isogeny.  

\end{proof}

%
%

The next step of the proof is to show that the support of $\shQ$ does not dominate
the image of $X$ in $B$. If we define $Z = g(X)$, which is a reduced and irreducible
subvariety of $B$, and if we consider $\shFm$, $\shG$, and $\shQ$ from now on as
coherent sheaves on $f(X) \subseteq Z \times F$, then the precise statement is the following.

\begin{lemma} \label{lem:support}
There is a nonempty Zariski-open subset $U \subseteq Z$ with the property that
$\Supp \shQ$ does not intersect $U \times F$.
\end{lemma}


\begin{proof}
For any point $b \in Z$, let $X_b = g^{-1}(b)$ and let $f_b \colon X_b \to F$ denote
the induced morphism, as in the following diagram:
\[
\begin{tikzcd}[column sep=large]
X_b \rar[hook] \dar{f_b} & X \dar{f} \arrow[bend left=60]{dd}{g} \\
F \rar[hook]{(i_b, \id)} \dar & Z \times F \dar{p} \\
\{b\} \rar[hook]{i_b} & Z
\end{tikzcd}
\]
Put $\shF_b = (f_b)_{\ast} \omega_{X_b}^{\tensor m}$. By the generic base change
theorem in Proposition \ref{prop:base-change}, there exists a point $b \in Z$
where the induced morphism $f_b \colon X_b \to F$ is smooth, and where the base
change morphism
\[
	(i_b, \id)^{\ast} \shFm \to \shF_b
\]
is an isomorphism. By induction, Theorem \ref{thm:main} holds on the abelian variety
$F$, and so $\shF_b$ becomes globally generated after an
isogeny on $F$. After performing a base change by such an isogeny (on the product $A
= B \times F$), we are therefore allowed to assume that $(i_b, \id)^{\ast} \shFm \cong
(f_b)_{\ast} \omega_{X_b}^{\tensor m}$ is globally generated; concretely, this means
that the evaluation morphism
\[
	H^0 \bigl( X_b, \omega_{X_b}^{\tensor m} \bigr) \tensor \shO_F \to
		(i_b, \id)^{\ast} \shFm	
\]
is surjective. At the same time, the morphism $g \colon X \to Z$ is smooth in a
neighborhood of the point $b \in Z$, and so invariance of plurigenera
\cite[Corollary~0.3]{Siu} tells us that the base change morphism
\[
	\iu_b(\gl \omXm) \to H^0 \bigl( X_b, \omega_{X_b}^{\tensor m} \bigr) 
\]
is surjective. Since $\gl \omXm = \pl \shFm$, we conclude that the natural morphism
\begin{equation} \label{eq:adjoint}
	\pu \pl \shFm \to \shFm
\end{equation}
is surjective at every point of the fiber $p^{-1}(b)$. By Nakayama's lemma, the same
thing is true in a Zariski-open neighborhood of $p^{-1}(b)$.

Now the identity $\pl \shG = \pl \shFm$ implies that the morphism in
\eqref{eq:adjoint} factors through the subsheaf $\shG$. Consequently, $\shG$ must be
equal to $\shFm$ -- and $\shQ$ must be trivial -- in a neighborhood of the fiber
$p^{-1}(b)$. But then $\Supp \shQ$ does not intersect the fiber $p^{-1}(b)$,
and so the assertion follows.
\end{proof}

It may seem that we are still very far away from proving Theorem \ref{thm:main},
because all we know up to now is that $\shQ = 0$ holds on an open subset of the form
$U \times F$. Furthermore, trying to prove a base change result for all fibers of $p$
is completely hopeless, especially since the varieties $X_b$ may be singular and of
the wrong dimension. But in fact -- and rather miraculously -- only one extra small
observation is needed to finish the proof. 

\begin{proposition} \label{prop:empty}
$V^0(A, \shQ)$ does not meet the image of $\pu \colon \Pic^0(B) \to \Pic^0(A)$.
\end{proposition}

\begin{proof}
Since $\Supp \shQ$ does not intersect the open set $U \times F$, the pushforward
sheaf $\pl \shQ$ is a torsion sheaf on $Z$. Now consider the exact sequence
\[
	0 \to \pl \shG \to \pl \shFm \to \pl \shQ \to R^1 \pl \shG.
\]
The first two sheaves in this exact sequence are equal, and $\pl \shQ$ is therefore
isomorphic to a subsheaf of $R^1 \pl \shG$. But $\shG$ is isomorphic to a direct
summand in
the pushforward of a canonical bundle, and so $R^1 \pl \shG$ is a \emph{torsion-free}
sheaf on $Z$ by Koll\'ar's theorem \cite[Theorem~2.1]{Kollar}. This forces $\pl \shQ = 0$, and
by the projection formula, we obtain
\[
	H^0(A, \shQ \tensor \pu \beta) \cong H^0(B, \pl \shQ \tensor \beta) = 0
\]
for every $\beta \in \Pic^0(B)$. 
\end{proof}

Recall from \eqref{eq:V0} that $V^0(A, \shQ) \subseteq V^1(A, \shG)$, and that every
irreducible component of $V^1(A, \shG)$ is of the form
\[
	\im \Bigl( \pu \colon \Pic^0(B) \to \Pic^0(A) \Bigr)
\]
where $p \colon A \to B$ is a surjective morphism of abelian varieties with connected
fibers, and $1 \leq \dim B \leq \dim A - 1$. (Recall that, by making a base change by
an isogeny, we had arranged that all irreducible components of $V^1(A, \shG)$ pass
through the origin.) We have proved enough to say that
$V^0(A, \shQ)$ must be empty: on the one hand, 
the intersection of $V^0(A, \shQ)$ with any positive-dimensional irreducible
component of $V^1(A, \shG)$ is empty by Proposition \ref{prop:empty}; on the other
hand, $V^1(A, \shG) = \{\OA\}$ implies that $V^0(A, \shQ) = \emptyset$ by
Proposition \ref{prop:origin}. The conclusion is that the $\mathrm{GV}$-sheaf $\shQ$ must be
zero. This means that $\shFm = \shG$, and so we have proved both Theorem \ref{thm:main}
and Theorem \ref{thm:main-summand}.

\subsection{Further consequences and remarks}\label{scn:consequences}

We first indicate how to prove Theorem \ref{thm:CJ}; except for the very last
assertion, it is of course a by-product of our proof of Theorem \ref{thm:main}.

\begin{proof}[Proof of Theorem~\ref{thm:CJ}]
By Theorem \ref{thm:main-summand}, $\fl \omXm$ is isomorphic to a direct summand in the
pushforward of a canonical bundle, and so by Theorem \ref{thm:main-1}, $\fl \omXm$ has
the Chen-Jiang decomposition property. Note that in a Chen-Jiang decomposition
\[
	\fl \omXm \cong \bigoplus_{i \in I} \bigl( \alpha_i \tensor \pu_i \shF_i \bigr),
\]
the line bundles $\alpha_i \in \Pic^0(A)$ are only determined up to elements of
$\pu_i \Pic^0(A_i)$. If $\varphi \colon A' \to A$ denotes the isogeny in
Theorem \ref{thm:main}, we have to explain why we can choose $\alpha_i$ in the
kernel of $\varphi^{\ast} \colon \Pic^0(A) \to \Pic^0(A')$. According to
Theorem \ref{thm:main}, 
\[
	\varphi^{\ast} \bigl( \fl \omXm \bigr)	
		\cong \bigoplus_{i \in I} \Bigl( \varphiu \alpha_i \tensor 
				(p_i \circ \varphi)^{\ast} \shF_i \Bigr)
\]
is globally generated; since $\shF_i \neq 0$, this can only be true if
\[
	0 \neq H^0 \bigl( A', \varphiu \alpha_i \tensor (p_i \circ \varphi)^{\ast} 
			\shF_i \bigr) 
		= H^0 \bigl( A_i, (p_i \circ \varphi)_{\ast} \varphiu \alpha_i \tensor 
			\shF_i \bigr).
\]
Consequently, $\varphiu \alpha_i \cong (p_i \circ \varphi)^{\ast} \beta_i$ for some
$\beta_i \in \Pic^0(A_i)$, which says exactly that $\alpha_i \tensor \pu_i
\beta_i^{-1}$ belongs to the kernel of $\varphiu \colon \Pic^0(A) \to \Pic^0(A')$.
\end{proof}

There are further well-known properties of (direct images of) canonical bundles that
Theorem \ref{thm:main-summand} immediately extends to the pluricanonical case as well:

\begin{corollary}
Let $f\colon X \to  A$ be a morphism from a smooth projective variety to an abelian variety, and let $m \ge 1$ be an integer. Then:
\begin{enumerate}
\item The cohomological support loci $V_{\ell}^i (A, f_* \omega_X^{\otimes m})$ are
finite unions of torsion subvarieties (= translates of abelian subvarieties by
torsion points).
\item The Fourier-Mukai transform $\derR \widehat{S} \derR \Delta \big(f_* \omega_X^{\otimes m}\big)$ is computed locally around each point 
by a linear (``derivative") complex of trivial vector bundles. 
\end{enumerate}
\end{corollary}

Indeed, these properties hold when $m=1$, the first by Simpson's theorem
\cite{Simpson}, and the second by \cite{ClH,LPS}, based of course on the main result of 
\cite{GL2}. (They are also subsumed in general results about the lowest non-zero term in the Hodge filtration on Hodge $\mathscr{D}$-modules in \cite{PS1}.)
Moreover for the proof of $(i)$ it may be helpful to refer to \cite[Lemma~10.3]{HPS}, which holds in greater generality 
than stated there, for all higher non-vanishing loci.

It is interesting to note that property $(i)$ above, although true for $V^0(X, \omega_X^{\otimes m})$, does not necessarily hold for higher cohomological support loci of 
either $\omega_X^{\otimes m}$ or $R^i f_* \omega_X^{\otimes m}$ with $i > 0$. In the first case, by this we mean
$$V^i (X, \omega_X^{\otimes m}) = \menge{\alpha \in \Pic^0(X)}{H^0(X, \omega_X^{\otimes m} \tensor \alpha) \neq 0}.$$
Here is one such situation:

\begin{example}
Let $C$ be a smooth curve of genus $g\geq 2$, and let $L$ be a line bundle on $C$ of degree $2g-1\leq d \leq 3g-3$.
We consider $f\colon X \to A$ to be the composition of the $\mathbf{P}^1$-bundle 
$$X = \mathbf{P}(L\oplus \shO_C)\; \stackrel{f}{\longrightarrow} \; C$$ 
with an Abel-Jacobi embedding of $C$ into $A = J(C)$.
We show that $V^2(X,\omega_X^{\otimes 2})$ and $V^1 \big(C,R^1f_*\omega_X^{\otimes 2}\big)$ are not unions of torsion translates of abelian subvarieties of $\Pic^0 (C)\cong \Pic^0(X)$. In the case of cohomological support loci of the form $V^i(X,\omega_X^{\otimes m})$ with $i, m \ge 2$, a 
similar example involving a projective bundle over an abelian variety containing an
elliptic curve was already communicated to us by T.~Shibata.

We have
$$\omega_X^{\otimes 2} \; \cong \; \shO_X(-4)\otimes f^*(\omega_C^{\otimes 2}\otimes L^{\otimes 2}),$$ 
from which one easily obtains
$$f_*\omega_X^{\otimes 2} =  0 \,\,\,\,\,\,{\rm and}\,\,\,\,\,\,
R^1f_*\omega_X^{\otimes 2} \; \cong \; \big( L^{-1}\oplus \shO_C\oplus L\big)\otimes \omega_C^{\otimes 2}.$$ 
Using the Leray spectral sequence we deduce the identifications
$$V^2(X, \omega_X^{\otimes 2})  = V^1 \big(A, R^1f_*\omega_X^{\otimes 2}\big)  = V^1(C, L^{-1}\otimes \omega_C^{\otimes 2})  = $$
$$= V^0(C, L\otimes \omega_C^{-1})  \cong  W_{d-2g+2}(C),$$ 
where the latter is the locus of degree $d-2g+2$ line bundles on $C$ admitting at least one non-zero global section. 
In the given range, this is a positive dimensional variety of general type. 
Note that for $d=3g-3$ the locus $V^1(X,\omega_X^{\otimes 2})$ is isomorphic to $W_{g-1}(C)$ as well.  
\end{example}

\section{Applications}

\subsection{Decomposition theorem and Iitaka fibration}

Let $X$ be a smooth projective variety with Kodaira dimension $\kappa(X)\geq 0$. In the discussion below, we may assume that we have
a regular model $I\colon X\rightarrow Z$ of the Iitaka fibration of $X$, with $Z$ smooth. 
We denote by $A_X$ and $A_Z$ the Albanese varieties of $X$ and $Z$, respectively.
There is a commutative diagram
$$
\begin{tikzcd}
X \rar{f} \dar{I} & A_X\dar{\psi} \\
Z\rar & A_Z 
\end{tikzcd}
$$
where $f$ is the Albanese map of $X$, and $\psi$ is the morphism induced by the
universal property of $A_X$.  Recall that $\psi$ has connected fibers, hence the
pullback morphism $\psi^* \colon \Pic^0
(Z) \to \Pic^0 (X)$ is injective; see for instance \cite[Proposition 2.1]{HP}.

The structure of the cohomology support loci of $\omega_X^{\otimes m}$ is described
by the following theorem \cite[Theorem~11.2]{HPS}; note that this needs $m \geq 2$.

\begin{theorem} \label{HPSv0}
Fix an integer $m\geq 2$. Then:
\begin{enumerate}
\item For every $\alpha \in \Pic^0(X)$ and every $\beta \in \Pic^0(Z)$, one has
\[
	\dim H^0 \bigl( X, \omXm \tensor \alpha \bigr) 
		= \dim H^0 \bigl( X, \omXm \tensor \alpha \tensor \psi^{\ast} \beta \bigr).
\]
\item The locus $V^0(X, \omega_X^{\otimes m})$ is the union of finitely many torsion
translates of the abelian subvariety $\psi^{\ast} \Pic^0(Z) \subseteq \Pic^0(X)$.
\end{enumerate}
\end{theorem}

We use this, together with Theorem \ref{thm:CJ}, in order to obtain a stronger
structural result for direct images of pluricanonical bundles via the Albanese map.


\begin{proof}[Proof of Theorem~\ref{thm:complete_intro}]
We know from Theorem \ref{thm:CJ} that $\fl \omXm$ admits a Chen-Jiang decomposition
\[
	\fl \omXm \cong \bigoplus_{i \in I} \bigl( \alpha_i \tensor \pu_i \shF_i' \bigr),
\]
with $p_i \colon A_X \to A_i$ surjective with connected fibers, and $\shF_i'$
a nonzero M-regular coherent sheaf on the abelian variety $A_i$. Moreover, the order
of the torsion line bundles $\alpha_i \in \Pic^0(X)$ is bounded independently of $m
\geq 2$. Lemma \ref{lem:CJ-V0} shows that
\[
	V^0(X, \omXm) = \bigcup_{i \in I} \alpha_i^{-1} \tensor \pu_i \Pic^0(A_i),
\]
and by Theorem \ref{HPSv0}, this set is contained in the union of finitely many
torsion translates of $\psi^{\ast} \Pic^0(Z)$. For each $i \in I$, we thus get a
factorization
\[
\begin{tikzcd}
	A_X \rar{\psi} \arrow[bend left=30]{rr}{p_i} & A_Z \rar{q_i} & A_i,
\end{tikzcd}
\]
and if we define $\shF_i = \qu_i \shF_i'$, which is a coherent sheaf on $A_Z$, we
have
\[
	\fl \omXm \cong \bigoplus_{i \in I} \bigl( \alpha_i \tensor \psi^{\ast} \shF_i \bigr).
\]
It remains to show that each $\shF_i$ satisfies $\mathrm{IT}_0$. Fix two line bundles
$\alpha \in \Pic^0(X)$ and $\beta \in \Pic^0(Z)$. By the projection formula,
\[
	\dim H^0 \Bigl( X, \omXm \tensor \alpha \tensor \psi^{\ast} \beta \Bigr)
	= \sum_{i \in I} \dim H^0 \Bigl( A_Z, \psi_{\ast} \bigl( 
			\alpha \tensor \alpha_i \bigr) 	\tensor \shF_i \tensor \beta \Bigr),
\]
and Theorem \ref{HPSv0} says that the quantity on the left is the same for 
every $\beta \in \Pic^0(Z)$. But each term in the sum on the right is an upper
semi-continuous function of $\beta$, and so all terms must be independent of $\beta$.
It follows, using Proposition \ref{locfreeh0}, that each coherent sheaf $\psi_{\ast}
\bigl( \alpha \tensor \alpha_i \bigr) \tensor \shF_i$ satisfies $\mathrm{IT}_0$ on
$A_Z$. Taking $\alpha = \alpha_i^{-1}$, we conclude that $\shF_i$ satisfies $\mathrm{IT}_0$.
\end{proof}

In the rest of the paper, $q(Y)$ denotes the irregularity of a variety $Y$, i.e. the dimension of the Albanese variety of 
(any resolution of singularities of) $Y$. The following is an important special case of Theorem \ref{thm:complete_intro}, which 
applies for instance to varieties of general type.

\begin{corollary}\label{equalq}
Let $X$ be a smooth projective variety with Iitaka fibration $I\colon X\rightarrow Z$ and Albanese map $f\colon X\rightarrow A$.
If $q(X)=q(Z)$, and for some $m\geq 2$ we have $f_*\omega_X^{\otimes m}\neq 0$, then 
\begin{enumerate}
\item $V^0(X, \omega_X^{\otimes m}) = \Pic^0 (X)$.
\item $f_*\omega_X^{\otimes m}$ satisfies $\mathrm{IT}_0$, so in particular it is an ample sheaf.
\end{enumerate}
\end{corollary}
\begin{proof}
Recall from \cite[Lemma~9.1]{HPS} that $f_*\omega_X^{\otimes m} \neq 0$ if and only
if $V^0(X, \omega_X^{\otimes m})$ is non-empty.  As $q(X)=q(Z)$, we have that $\psi^*$ is an isomorphism, and so Theorem \ref{HPSv0} immediately implies (i), while Theorem \ref{thm:complete_intro} implies that  $f_*\omega_X^{\otimes m}$ satisfies $\mathrm{IT}_0$.  It is therefore ample
by \cite[Corollary~3.2]{Debarre} (combined with \cite[Proposition 2.13]{PP4}).
\end{proof}

\begin{remark}
Note however that the corollary above could in fact be seen as a toy case of the proof of Theorem \ref{thm:complete_intro}. 
Indeed, under its assumptions one can directly combine Theorem \ref{HPSv0} and Proposition \ref{locfreeh0} to obtain 
the conclusion.
\end{remark}

\subsection{Effective generation for pluricanonical bundles}\label{scn:effective}
In this section we use Theorem \ref{thm:complete_intro} in order to give effective bounds on the
generation of pluricanonical bundles on irregular varieties of general type with
at most canonical singularities. This extends, or provides variants of, results in
\cite{CH,PP4,PP2} in the singular case.

We will need the following simple surjectivity criterion for adjoint morphisms.

\begin{lemma}\label{prop:adjsur}
Let $f\colon X\rightarrow Y$ be a projective morphism of complex varieties. Let $\shF$ and $\shG$ be coherent sheaves on $X$ 
and $Y$, respectively, together with a surjective morphism $\varphi\colon \shG\rightarrow f_*\shF$. Then the adjoint morphism 
$f^*\shG\rightarrow \shF$ is surjective on $f^{-1}(W)$, where $W\subseteq Y$ is the open subset of points $y\in Y$ satisfying:
\begin{enumerate}
\item $\shF_{|f^{-1}(y)}$ is globally generated, and 
\item  the natural morphism $f_*\shF\otimes \mathbf{C}(y)\rightarrow 
H^0\big(f^{-1}(y),\shF_{|f^{-1}(y)}\big)$ is surjective.  
\end{enumerate}  
\end{lemma}
\begin{proof}
The adjoint morphism of $\varphi$ factors as $f^*\shG\rightarrow f^*f_*\shF\rightarrow \shF$. As $f^*\varphi$ is surjective, it is enough to show that $\psi\colon f^*f_*\shF\rightarrow \shF$ is surjective when restricted to $f^{-1}(y)$, for all $y\in W$. Using $(i)$ and $(ii)$, we find that the morphisms 
$\psi_{|f^{-1}(y)}$ are compositions of surjective morphisms. 
\end{proof}

In what follows we will work with a normal variety $Y$ such that $K_Y$ is
$\mathbf{Q}$-Cartier. We denote by $N_Y$ the \emph{index of} $Y$, meaning the smallest
positive integer $N$ for which $N K_Y$ is Cartier. The very ampleness statement in part $(ii)$ below is quite technical; we only use it later in the generically finite case, but we include it since it may prove useful in other applications.
 
\begin{theorem}\label{gg-index}
Let $Y$ be a normal projective variety with canonical singularities, $I\colon Y\rightarrow Z$ a smooth model of its Iitaka fibration, and  $g\colon Y\rightarrow A$ its Albanese map, with fiber $Y_s$ over $s \in A$. Assume that $q(Y)=q(Z)$ (for instance $Y$ 
could be of general type), and that there is an integer $r=kN_Y \ge 2$, with $k$ a positive integer, and a nonempty open subset 
$W_r \subset A$ consisting of points $s$ such that: 
\begin{itemize}
\item $\shO_{Y_s}(rK_Y)$ is globally generated;\\
\item the natural morphism $g_*\shO_Y(rK_Y)\otimes \mathbf{C}(s) \, \rightarrow  \, 
H^0\big(Y_s, \shO_{Y_s}(rK_Y)\big)$ is surjective. 
\end{itemize}
Then:
\begin{enumerate}
\item $\shO_Y(2 r K_Y)\otimes \alpha$ is globally generated on $g^{-1}(W_r)$ for every $\alpha \in \Pic^0(Y)$.\\
\item Assuming in addition that:\\
\begin{itemize}
\item $\shO_{Y_s}(2rK_Y)\otimes \sI_{y|Y_s}$ is globally generated for every $y\in g^{-1}(W_r)$;\\
\item the natural morphism $$g_*\big(\shO_Y(2rK_Y) \otimes \sI_y\big) \otimes \mathbf{C}(s) \, \rightarrow  \, 
H^0\left(Y_s, \shO_{Y_s}(2rK_Y)\otimes \sI_{y|Y_s}\right)$$ is surjective for every $y\in g^{-1}(W_r)$,\\
\end{itemize}
then $\shO_Y(3rK_Y)\otimes \alpha$ is very ample on $g^{-1}(W_r)$ for all  $\alpha \in \Pic^0(Y)$.
\end{enumerate}
\end{theorem}
 
 \begin{proof} 
To see $(i)$, note that since $Y$ has canonical singularities, we have 
$$\varphi_*\omega_X^{\otimes m} \cong \shO_Y(mK_Y)\quad \mbox{for all}\quad m\geq 1,$$
where $\varphi\colon X\rightarrow Y$ is a resolution of $Y$. 
Note that the composition $f = g \circ \varphi \colon X\rightarrow A$ is the Albanese map of $X$.

We fix $\beta\in \Pic^0(A)$ such that $\beta^{\otimes 2}\cong \alpha$, and identify $\Pic^0 (Y)$ and $\Pic^0 (A)$ 
via $g^*$. The hypotheses imply that
$$f_*\omega_X^{\otimes r}\otimes \beta \, \cong \, g_*\shO_Y(rK_Y) \otimes \beta \cong \,  
g_*(\shO_Y(rK_Y) \otimes \beta)  \neq 0,$$ 
and by Corollary \ref{equalq} we obtain that the sheaf $ g_*(\shO_Y(rK_Y)\otimes \beta)$ 
satisfies $\mathrm{IT}_0$ on $A$. By \cite[Proposition 2.13]{PP4} it follows that it is continuously globally generated, i.e. there exists an integer $M$ such that for general line bundles $\alpha_1,\ldots ,\alpha_M \in \Pic^0(A)$ the sum of twisted
evaluation maps
$$\bigoplus_{i=1}^M H^0 \big(A,g_*(\shO_Y(rK_Y) \otimes \beta) \otimes \alpha_i\big)\otimes \alpha_i^{-1} \, \rightarrow \, g_*(\shO_Y(rK_Y) \otimes \beta)$$ 
is surjective. Using Lemma \ref{prop:adjsur}, it follows that the adjoint morphism 
$$\bigoplus_{i=1}^M H^0(Y,\shO_Y(rK_Y) \otimes \beta \otimes g^*\alpha_i)\otimes g^*\alpha^{-1}_i \, \rightarrow \, \shO_Y(rK_Y) \otimes \beta$$ 
is surjective on $g^{-1}(W_r)$ for such $\alpha_i$, so that $\shO_Y(rK_Y) \otimes \beta$ is continuously globally generated on $g^{-1}(W_r)$. By \cite[Proposition 2.12]{PP4} we conclude that 
$$\sO_Y(2rK_Y) \otimes \alpha \cong (\shO_Y(rK_Y)\otimes \beta)^{\otimes 2}$$ 
is generated by global sections on $g^{-1}(W_r)$.

We next show $(ii)$. We approach very ampleness by means of the following well-known criterion: a line bundle $L$ on a variety $Z$ is very ample if and only if $L\otimes \shI_y$ is globally generated for all $y \in Z$. 

As we have seen that $\shO_Y(rK_Y)$ is continuously globally generated on $g^{-1}(W_r)$, again by \cite[Proposition 2.12]{PP4} it suffices to show that 
$$\shO_Y(2rK_Y)\otimes \shI_y\otimes \alpha$$
is also continuously globally generated on $g^{-1}(W_r)$, for all $y\in g^{-1}(W_r)$ and $\alpha \in \Pic^0(Y)$.
Using an argument completely similar to that in $(i)$ (in which the two extra hypotheses come into play), 
this is the case as long as the sheaves 
$g_*(\shO_Y(2rK_Y)\otimes \shI_y\otimes \alpha)$ satisfy $\mathrm{IT}_0$ for all $y\in g^{-1}(W_r)$ and 
$\alpha\in \Pic^0 (A)$. 

Pushing forward  the obvious short exact sequence on $Y$, we obtain an exact sequence:
\begin{equation}\label{eqonA}
0 \, \rightarrow \, g_*\big(\shO_Y(2rK_Y)\otimes \shI_y\otimes \alpha\big) \, \rightarrow \, g_*\big(\shO_Y(2rK_Y)\otimes \alpha\big) \, \rightarrow \, G_{y,\alpha} \, \rightarrow \, 0,
\end{equation}
where $G_{y,\alpha}$ is a sheaf that embeds in $g_*\big( \shO_y(2rK_Y) \otimes \alpha\big) \cong g_*\shO_y$.
As above, 
$$g_*\big(\shO_Y(2rK_Y)\otimes \alpha\big) \cong f_*\omega_X^{\otimes 2r} \otimes \alpha$$ 
is nonzero, and hence satisfies $\mathrm{IT}_0$ by Corollary \ref{equalq}. This immediately yields  
\begin{equation}\label{vanishings}
H^i\big(A,g_*\big(\shO_Y(2rK_Y)\otimes \shI_y\otimes \alpha\big) \otimes \gamma\big) \, = \, 0 \quad 
\mbox{ for all }\quad i\geq 2, \quad \gamma \in \Pic^0(A).
\end{equation}
On the other hand, since $\shO_Y(2rK_Y) \otimes \gamma$ is globally generated on $g^{-1}(W_r)$ for all 
$\gamma\in \Pic^0(A)$, we obtain the surjectivity of 
$$H^0\big(A,g_*(\shO_Y(2rK_Y) \otimes \alpha) \otimes \gamma \big) \, \rightarrow \, H^0\big(A,g_*( \shO_y(2rK_Y) \otimes \alpha) \otimes \gamma \big)$$
for all $y\in g^{-1}(W_r)$ and $\gamma \in \Pic^0(A)$, 
which in turns yields the surjectivity of 
$$H^0\big(A,g_*(\shO_Y(2rK_Y) \otimes \alpha) \otimes \gamma \big) \, \rightarrow \,  H^0(A,G_{y,\alpha}\otimes \gamma)$$
for such $y$ and $\gamma$.
We conclude that the vanishing in \eqref{vanishings} also holds for $i=1$, which finishes the proof. 
 \end{proof}

In the general type case, we spell out a couple of special instances in which Theorem \ref{gg-index} applies. 
The first is concerned with varieties of maximal Albanese dimension, and extends
\cite[Theorem 5.1]{PP4} and \cite[Theorem 6.7]{PP2}\footnote{Note that the statement
of the theorem in \emph{loc. cit.} left out the general type hypothesis.} to the singular setting; it is essentially Corollary \ref{cor:mad} in the Introduction. 
  
\begin{corollary}\label{cor:genfinite}
Let $Y$ be a normal projective variety of general type with at worst canonical singularities, whose Albanese map 
$g\colon Y\rightarrow A$ is generically finite onto its image. Denote by ${\rm Exc}(g)$ the union of the positive dimensional fibers of 
$g$.  If $k$ is a positive integer such that $r = k N_Y \ge 2$, then for every $\alpha \in \Pic^0(Y)$:
\begin{enumerate}
\item $\shO_Y(2r K_Y)\otimes \alpha$ is globally generated away from ${\rm Exc}(g)$.
\item $\shO_Y(3rK_Y)\otimes \alpha$ is very ample away from ${\rm Exc}(g)$.
\end{enumerate}
\end{corollary}
  
As mentioned in the Introduction, when $Y$ is smooth (and so $N_Y =1$), stronger results are proved in \cite{JLT}. It would be interesting to see whether further 
results along those lines can be proved in the singular setting, in combination with the results in this paper.
  
\begin{remark}
As Jungkai Chen points out, it would also be interesting to have a statement in terms of reflexive powers of $\omega_Y$, independent of $N_Y$. We are unable to do this at the moment, the main reason being that in \cite[Proposition 2.12]{PP4} one of the factors is required to be a line bundle. One can however do better than Corollary \ref{cor:genfinite} sometimes, in the following sense. Denote by 
$$\omega_Y^{[m]}:=\big(\omega_Y^{\otimes m}\big)^{**}\cong \shO_Y(mK_Y)$$ 
the $m$-th reflexive power of the canonical sheaf $\omega_Y$. Then, say if $N_Y \ge 2$, the sheaves
$$\shO_Y(N_YK_Y)\otimes \omega_Y^{[m]} \otimes \alpha$$ 
are globally generated away from ${\rm Exc}(g)$ for every $\alpha \in \Pic^0(Y)$ and $m\geq 2$. Indeed, if $\varphi\colon X\rightarrow Y$ is a resolution of singularities, then there are isomorphisms $\varphi_*\omega_X^{\otimes m} \cong \omega_Y^{[m]}$ for all $m\geq 1$ (\cite[Proposition~10.8]{Ko}), hence as in the proof 
of Theorem \ref{gg-index} the sheaf $g_* \omega_Y^{[m]} \otimes \beta$ is continuously globally generated for every $\beta \in \Pic^0 (A)$.  
\end{remark}

The second is concerned with the case when the fibers are positive dimensional; cf. \cite[Theorem~4.2]{CH} for a result 
in this direction in the smooth case.

\begin{corollary}\label{cor:fibration}
Let $Y$ be a normal projective variety of general type with at worst canonical singularities, and  
$g\colon Y\rightarrow A$ its Albanese map, with fiber $Y_s$ over $s \in A$.
Let $r = k N_Y$, with $k$ a positive integer, and consider an open nonsingular subset 
$W_r \subset g(Y)$ such that: 
\begin{enumerate}
\item $g$ is flat over $W_r$.
\item $Y_s$ is a normal variety of general type with at worst canonical singularities for all $s\in W_r$.
\item $\shO_{Y_s}(rK_{Y_s})$ is globally generated for all $s\in W_r$.
\end{enumerate}
If $r \ge 2$, then $\shO_Y(2rK_Y)\otimes \alpha$ is globally generated on $g^{-1}(W_r)$ for  every $\alpha\in \Pic^0(Y)$.
\end{corollary}
\begin{proof} 
Since both sheaves are reflexive and coincide in codimension one, we have an isomorphism
$$\shO_{Y_s}(r K_Y) \, \cong \, \shO_{Y_s}(r K_{Y_s}) \quad \mbox{for all}\quad s\in A,$$
so that the first condition in Theorem \ref{gg-index} is satisfied. 
On the other hand, the hypotheses ensure that over $W_r$ one can apply Kawamata's theorem \cite[Theorem 6]{Kaw:deformations} to obtain that the plurigenera
 $$P_{r}(Y_s) \; = \;  h^0\big(Y_s, \shO_{Y_s}(r K_{Y_s})\big)$$ 
 are constant. By Grauert's Theorem this implies that the second condition in Theorem \ref{gg-index} 
 is satisfied.
  \end{proof}

\bigskip
\footnotesize
\noindent\textit{Acknowledgments.}
We thank Jungkai Chen, Daniel Greb, Christopher Hacon, Zhi Jiang, S\'andor Kov\'acs, Vlad Lazi\'c and Takahiro Shibata for answering our questions and for useful discussions. We also thank the referees for suggesting some expository improvements.
During the preparation of the paper LL benefited from a visit to the Mathematical Institute of the University of Bonn. He thanks Daniel Huybrechts and Luca Tasin for useful conversations and their hospitality. 
LL acknowledges the support of the Simons Foundation.
MP was partially supported by NSF grant DMS-1405516 and by a Simons Fellowship. He thanks Stony Brook University for hospitality during a visit when this project was started. CS was partially supported by NSF grant DMS-1404947 and by a Centennial Fellowship from the American Mathematical Society.
 
%

\begin{thebibliography}{EMS}



%





 



	 

\bibitem[BCHM]{BCHM}
	 Birkar, C., Cascini, P.,  Hacon, Ch., McKernan, J.:
	Existence of minimal models for varieties of log general type.
	J. Amer. Math. Soc. \textbf{23}, no.~2, 405--468 (2010)
	 
	 \bibitem[CP]{CP}
Cao, J., P\u{a}un, M.: 
Kodaira dimension of algebraic fiber spaces over abelian varieties.
 Invent. Math. \textbf{207}, no.~1, 345--387 (2017)
 
  \bibitem[ChH]{CH}
	Chen, J. A.,  Hacon, Ch.:
	Linear series of irregular varieties.
In: Algebraic Geometry in East Asia (Kyoto, 2001), World Sci. Publ., River Edge NJ, 143--153 (2002)
 
 \bibitem[CJ]{ChenJiang:positivity}
	Chen, J. A., Jiang, Z.:
Positivity in varieties of maximal Albanese dimension.
J. Reine Angew. Math. \textbf{736}, 225--253 (2018)

\bibitem[ClH]{ClH}
	Clemens, H., Hacon, Ch.:
	Deformations of the trivial line bundle and vanishing theorems.
	Amer. J. Math. \textbf{124}, no.~4, 769--815 (2002) 
	
	
 
\bibitem[D]{Debarre}
Debarre, O.:
On coverings of simple abelian varieties.
Bull. Soc. Math. France \textbf{134}, no.~2, 253--260 (2006)
 
\bibitem[GL1]{GL1}
Green, M.,  Lazarsfeld, R.:
	 Deformation theory, generic vanishing theorems, and some conjectures of Enriques, Catanese and Beauville.
	Invent. Math. \textbf{90}, no.~2, 389--407 (1987)
	
	\bibitem[GL2]{GL2}
Green, M., Lazarsfeld, R.:
	Higher obstructions to deforming cohomology groups of line bundles.
	J. Amer. Math. Soc. \textbf{4}, no.~1, 87--103 (1991)


\bibitem[Gro]{EGA}
Grothendieck, A.
\'El\'ements de g\'eom\'etrie alg\'ebrique. IV. \'Etude locale des sch\'emas et des morphismes de sch\'emas. II.
	Inst. Hautes \'Etudes Sci. Publ. Math., no.~24, pp. 231 (1965)
	
	
\bibitem[H]{Hacon:GV}
	Hacon, Ch.:
	A derived category approach to generic vanishing.
	J. Reine Angew. Math. \textbf{575}, 173--187 (2004) 
	
 
\bibitem[HP]{HP}
Hacon, Ch.,  Pardini, R.: 
On the birational geometry of varieties of maximal Albanese dimension.
J. Reine Angew. Math. \textbf{546}, 177--199 (2002)

\bibitem[HPS]{HPS}
	Hacon, Ch., Popa, M., Schnell, Ch.:
	 Algebraic fiber spaces over abelian varieties: around a recent theorem by Cao and P{\u{a}}un.
	In: Local and Global Methods in Algebraic Geometry, Contemporary Mathematics \textbf{712}, Amer. Math. Soc.,  143--195 (2018)
	 
	 
\bibitem[JLT]{JLT}
Jiang, Z., Lahoz, M., Tirabassi, S.:
	On the Iitaka fibration of varieties of maximal Albanese dimension.
	Int. Math. Res. Not. IMRN, no.~13, 2984--3005 (2013)
	
	
\bibitem[Kaw]{Kaw:deformations}
	Kawamata, Y.:
 Deformations of canonical singularities.
J. Amer. Math. Soc. \textbf{12}, no.~1, 85--92 (1999)
	
	\bibitem[Kol]{Kollar}
	Koll{\'a}r, J.:
	Higher direct images of dualizing sheaves I.
Ann. of Math. (2) \textbf{123}, no.~1, 11--42 (1986)


\bibitem[Kov]{Ko}
Kov\'acs, S.:
	Rational singularities.
Preprint arXiv:1703.02269v7 (2019)

\bibitem[Lai]{Lai}
Lai, C.-J.:
Varieties fibered by good minimal models.
Math. Ann. \textbf{350}, no.~3, 533--547 (2011)

\bibitem[LPS]{LPS}
Lazarsfeld, R., Popa, M., Schnell, Ch.:
	Canonical cohomology as an exterior module.
	Pure Appl. Math. Q. \textbf{7}, no.~4, Special Issue: In memory of Eckart Viehweg, 1529--1542 (2011)
	
 \bibitem[PP1]{PP4}
Pareschi, G., Popa, M.: 
Regularity on abelian varieties I.
J. Amer. Math. Soc. \textbf{16}, no.~2, 285--302 (2003)	
	
	\bibitem[PP2]{PP5}
	Pareschi, G., Popa, M.:
	$M$-regularity and the Fourier-Mukai transform.
	Pure and Applied Math. Q., special issue in honor of F. Bogomolov \textbf{4}, no.~3, 587--611 (2008)
	 
	 \bibitem[PP3]{PP1}
	Pareschi, G., Popa, M.: 
	Strong generic vanishing and a higher-dimensional Castelnuovo-de Franchis inequality.
Duke Math. J. \textbf{150}, no.~2, 269--285 (2009)

\bibitem[PP4]{PP3}
Pareschi, G., Popa, M.:
 GV-sheaves, Fourier-Mukai transform, and generic vanishing.
 Amer. J. Math. \textbf{133}, no.~1, 235--271 (2011) 

\bibitem[PP5]{PP2}
	Pareschi, G., Popa, M.: 
Regularity on abelian varieties III: relationship with generic vanishing and applications.	
In: Grassmannians, moduli spaces and vector bundles, Clay Math. Proc. \textbf{14}, Amer. Math. Soc., Providence RI, 141--167 (2011)  
 
 \bibitem[PPS]{PPS}
Pareschi, G., Popa, M., Schnell, Ch.: 
Hodge modules on complex tori and generic vanishing for compact K\"ahler manifolds.
Geometry \& Topology \textbf{21}, 2419--2460 (2017)
 
 \bibitem[PS1]{PS1}
	Popa, M., Schnell, Ch.: 
	Generic vanishing theory via mixed Hodge modules.
    Forum Math. Sigma \textbf{1}, e1, 60 (2013)
	 
	 \bibitem[PS2]{pluricanonical}
Popa, M., Schnell, Ch.:
	 On direct images of pluricanonical bundles.
Algebra Number Theory \textbf{8}, no.~9, 2273--2295 (2014)

\bibitem[Sh]{Shibata2}
Shibata, T.:
 On generic vanishing for pluricanonical bundles.
 Michigan Math. J. \textbf{65}, no.~4, 873--888 (2016)
  

 
\bibitem[Sim]{Simpson}
Simpson, C.: 
Subspaces of moduli spaces of rank one local systems.
Ann. Sci. \'Ecole Norm. Sup. (4) \textbf{26}, no.~3, 361--401 (1993)    
 
 \bibitem[Siu]{Siu}
Siu, Y.-T.:
Extension of twisted pluricanonical sections with plurisubharmonic weight and invariance of semipositively twisted plurigenera for manifolds not necessarily of general type.
	In: Complex Geometry (G{\"o}ttingen, 2000), Springer, Berlin, 223--277 (2002)
 


 

\bibitem[V]{Viehweg}
	Viehweg, E.:
	Weak positivity and the additivity of the Kodaira dimension for certain fibre spaces.
In: Algebraic Varieties and Analytic Varieties (Tokyo, 1981), Adv. Stud. Pure Math. \textbf{1}, North-Holland, Amsterdam, 329--353 (1983)
	   
  
  


 

	 


 
  
\end{thebibliography}

\end{document}